\documentclass[psamsfonts]{amsart}

\usepackage{amssymb,amsmath}
\usepackage{amsthm}
\usepackage{mathrsfs}
\usepackage{enumerate, indentfirst}

\usepackage[fleqn,tbtags]{mathtools}
\usepackage{marginnote}
\usepackage{todonotes}
\usepackage{multicol}
 \newtheorem{theorem}{Theorem}[section]
 \newtheorem{corollary}[theorem]{Corollary}
 \newtheorem{lemma}[theorem]{Lemma}
 \newtheorem{proposition}[theorem]{Proposition}

 \theoremstyle{definition}

 \theoremstyle{remark}
 
  \numberwithin{equation}{section}

\renewcommand{\phi}{\varphi}
\renewcommand{\theta}{\vartheta}
  
\DeclareMathOperator{\tform}{\mathfrak{t}}

\DeclareMathOperator{\wform}{\mathfrak{w}}

\DeclarePairedDelimiterX\sipt[2]{(}{)_{\tform}}{#1\,\delimsize\vert\,#2}
\DeclarePairedDelimiterX\sipv[2]{(}{)_{v}}{#1\,\delimsize\vert\,#2}
\DeclarePairedDelimiterX\sipw[2]{(}{)_{w}}{#1\,\delimsize\vert\,#2}

\newcommand{\pair}[2]{\left(\begin{array}{c}\!\!#1\!\!\\ \!\!#2\!\!\end{array}\right)}

\newcommand{\abs}[1]{\lvert#1\rvert}
\newcommand{\dupN}{\mathbb{N}}

\newcommand{\seq}[1]{(#1_{n})_{n\in\dupN}}

\newcommand{\dom}{\operatorname{dom}}

\newcommand{\ran}{\operatorname{ran}}

\newcommand{\hil}{\mathcal{H}}

\newcommand{\kil}{\mathcal{K}}

\DeclarePairedDelimiterX\sip[2]{(}{)}{#1\,\delimsize\vert\,#2}
\DeclarePairedDelimiterX\siptilde[2]{(}{)_{\!_{\widetilde{A}}}}{#1\,\delimsize\vert\,#2}
\DeclarePairedDelimiterX\sipf[2]{(}{)_{f}}{#1\,\delimsize\vert\,#2}
\DeclarePairedDelimiterX\sipg[2]{(}{)_{g}}{#1\,\delimsize\vert\,#2}
\DeclarePairedDelimiterX\siptw[2]{(}{)_{\tform+\wform}}{#1\,\delimsize\vert\,#2}
\DeclarePairedDelimiterX\set[2]{\{}{\}}{#1\,\delimsize\vert\,#2}
\DeclarePairedDelimiterX\dual[2]{\langle}{\rangle}{#1,#2}
\DeclarePairedDelimiterX\sipa[2]{(}{)_{\!_A}}{#1\,\delimsize\vert\,#2}
\DeclarePairedDelimiterX\sipc[2]{(}{)_{\!_C}}{#1\,\delimsize\vert\,#2}
\DeclarePairedDelimiterX\sipab[2]{(}{)_{\!_{A+B}}}{#1\,\delimsize\vert\,#2}
\DeclarePairedDelimiterX\sipb[2]{(}{)_{\!_B}}{#1\,\delimsize\vert\,#2}

\newcommand{\operator}[2]{\left(\begin{array}{cc}\!\! I &  #1\!\!\\ \!\! #2& I\!\!\end{array}\right)}

\allowdisplaybreaks
\newcommand{\kismatrix}[4]{\begin{psmallmatrix} #1 &  #2\\ #3& #4\end{psmallmatrix}}
\newcommand{\kisoperator}[2]{\begin{psmallmatrix} I &  #1\\ #2 & I\end{psmallmatrix}}
\newcommand{\kispair}[2]{\begin{psmallmatrix} #1 \\  #2\end{psmallmatrix}}
\begin{document}
\title{On the adjoint of Hilbert space operators}

\author[Z. Sebesty\'en]{Zolt\'an Sebesty\'en}

\address{%
Department of Applied Analysis and Computational Mathematics,\\ E\"otv\"os L. University,\\ P\'azm\'any P\'eter s\'et\'any 1/c,\\ Budapest H-1117,\\ Hungary}

\email{sebesty@cs.elte.hu}

\author[Zs. Tarcsay]{Zsigmond Tarcsay}

\address{%
Department of Applied Analysis,\\ E\"otv\"os L. University,\\ P\'azm\'any P\'eter s\'et\'any 1/c,\\ Budapest H-1117,\\ Hungary}

\email{tarcsay@cs.elte.hu}




\begin{abstract}
In general, it is a non trivial task to determine the adjoint $S^*$ of an unbounded operator $S$ acting between two Hilbert spaces. We provide necessary and sufficient conditions for a given operator $T$ to be identical with $S^*$. In our considerations, a central role is played by the operator matrix $M_{S,T}=\kisoperator{-T}{S}$.  Our approach has several consequences such as characterizations of  closed, normal,  skew- and selfadjoint, unitary and orthogonal projection operators in real or complex Hilbert spaces. We also give a self-contained proof of the fact that $T^*T$ always has a positive selfadjoint extension. 
\end{abstract}

\maketitle

\section{Introduction}
The notion of adjoint operator of a densely defined linear operator $S$ acting between the (real or complex) Hilbert spaces $\hil$ and $\kil$ is originated by J. von Neumann \cite{vonNeumann} and is determined as an operator $S^*$ from $\kil$ into $\hil$ having domain
\begin{gather*}
\dom S^*=\set{k\in\kil}{\sip{Sh}{k}=\sip{h}{k^*}\textrm{~for some $k^*$, for all $h\in\dom S $}},
\end{gather*}
and  acting by 
\begin{align*}
S^*k:=k^*,\qquad k\in\dom S^*.
\end{align*}
Here the uniqueness of $k^*$ is guaranteed by  density of the domain $\dom S$ of $S$.  Nevertheless, it is a non-trivial task to determine the adjoint $S^*$ of $S$, that is, to describe the domain $\dom S^*$ of $S^*$ explicitly and to specify the action of $S^*$ on elements of $\dom S^*$. Clearly, $S$ and its adjoint $S^*$ fulfill the ``adjoining identity'' 
\begin{align*}
\sip{Sh}{k}=\sip{h}{S^*k},\qquad h\in\dom S, k\in\dom S^*,
\end{align*}
that is to say, $T=S^*$ is a linear operator from $\kil$ into $\hil$ which satisfies 
\begin{equation}\label{E:SwedgeT}
\sip{Sh}{k}=\sip{h}{Tk}, \qquad  h\in\dom S, k\in\dom T.
\end{equation} 
However, in order to have $T=S^*$ it is not enough to demand  that $T$ satisfies \eqref{E:SwedgeT}. For instance, every symmetric operator $S$, without  being selfadjoint, satisfies \eqref{E:SwedgeT} with $T=S$. 

In the present paper we are particularly interested in pairs of (not necessarily densely defined) linear operators $S$ and $T$ from $\hil$ into $\kil$, and $\kil$ into $\hil$, respectively, which fulfill identity \eqref{E:SwedgeT}. We adapt the terminology of M. H. Stone \cite{Stone} and  say  that $S$ and $T$ are \emph{adjoint to each other} if they satisfy \eqref{E:SwedgeT} and  write 
\begin{align*}
S\wedge T,
\end{align*}
in that case (cf. also \cite{Popovici,Characterization,SZTZS2015}). Our main purpose in this paper is to provide a method to verify whether the operators $S$ and $T$ under the weaker condition $S\wedge T$ satisfy the stronger property $S^*=T$, or the much stronger one of being adjoint \emph{of} each other, i.e.,  $S^*=T$ and $T^*=S$. In this direction our main results are Theorem \ref{T:theorem22} and Theorem \ref{T:theorem31} which give necessary and sufficient conditions by means of the operator matrix $$M_{S,T}:=\operator{-T}{S}$$ acting on the product Hilbert space $\hil\times \kil.$ 

A remarkable advantage of our treatment is that no density assumption on the domains of the operators $S,T$ is imposed. On the contrary, densely definedness is just achieved as a consequence of the other conditions.  Furthermore, the results are not limited to \emph{complex} Hilbert spaces, they will remain valid in real spaces as well. This in turn allows to extend von  Neumann's results characterizing skew-adjoint, selfadjoint,  and positive selfadjoint operators, to real Hilbert space setting. 

The paper is organized as follows. In section 2 we discuss the question whether a given operator $T$ is identical with the adjoint of another operator $S$.   The main result in this direction is Theorem \ref{T:theorem22} that gives an answer  by means of the range of the operator matrix $M_{S,T}$.  This result will be extensively used  throughout. In Theorem  \ref{T:theorem31} of section 3 a full description of operators which are adjoint of each others is established. This is a sharpening of Theorem 3.4 in \cite{Popovici}. We also offer a  ``dual'' version of the Hellinger--Toeplitz thereom  which concerns full range operators that are adjoint to each other. In section 4 we consider sums and products of   linear operators. Our purpose here is to describe the situations in which for given two operators $R,S$ equalities $(R+S)^*=R^*+S^*$ and $(RS)^*=S^*R^*$ hold. In \cite{Sebestyen83a} the first named author offered a metric characterization of the range of the adjoint $S^*$ of a densely defined linear operator $S$, namely,  $\ran S^*$ is described by consisting of those vectors $z$ which fulfill a Schwarz type inequality 
\begin{align*}
\abs{\sip{x}{z}}\leq M_z\cdot \|Sx\|,\qquad x\in\dom S
\end{align*}
with some nonnegative constant $M_z$. In section 5 we improve this result to describe the range of an operator $T$ which is adjoint to an operator $S$. 
In sections 6 and 7 we deal with skew-adjoint, selfadjoint, and positive selfadjoint operators and characterize them among skew-symmetric, symmetric and positive symmetric operators. Instead of using the defect index theory developed by J. von Neumann, our method involves the range of $M_{S,S}$. In Theorem \ref{T:T*Text} we also present a new proof of the fact that $T^*T$ always has a positive selfadjoint extension (see \cite{TadjointT}). In  section 8  we characterize  densely defined closed operators. The main result of the section establishes also a converse to  Neumann's classical result:  $T^*T$ and $TT^*$ are both selfadjoint operators if and \emph{only if} $T$ is densely defined and closed (see also \cite{SZ-TZS:reversed}). Finally, in section 9 we obtain some characterizations of normal, unitary and orthogonal projection operators.

\section{Characterization of the adjoint of a linear operator}

Let $\hil$ and $\kil$ be real or complex Hilbert spaces and let $S$ be a not necessarily densely defined or closed linear operator between them. The problem mentioned in the introduction consists of the identification of the adjoint $S^*$ of $S$ (provided that $S$ is densely defined). Doing so we start by fixing another linear operator $T$ between $\kil$ and $\hil$ satisfying \eqref{E:SwedgeT}, that is, $S$ and $T$ are adjoint to each other. As a  main tool for our investigations we introduce  the operator matrix $$M_{S,T}:=\operator{-T}{S}$$ associated with $S$ and $T$, which is defined on $\dom M_{S,T}=\dom S\times\dom T$ by the correspondence
\begin{equation*}
\pair{h}{k}\mapsto\pair{h-Tk}{Sh+k},\qquad h\in\dom S, k\in\dom T.
\end{equation*}
The importance of the role of the operator matrix $M_{S,T}$ initiates the recent papers of the authors \cite{Popovici,PopSZTZS, Characterization,Charess,SZTZS2015}.
The ``flip'' operator $W:\kil\times\hil\to\hil\times \kil$ will also be useful for our analysis, which is defined as follows (see \cite{Birman}):
\begin{equation*}
W(k,h)=(-h,k),\qquad h\in\hil, k\in\kil.
\end{equation*}
The symbols $pr_\hil$ and $pr_\kil$ stand for the canonical projections of the product Hilbert space $\hil\times\kil$ onto $\hil$ and $\kil$, respectively,  which are defined accordingly by the correspondences
\begin{equation}
pr_{\hil}(h,k)=h,\qquad pr_{\kil}(h,k)=k,\qquad h\in\hil, k\in\kil.
\end{equation}
The graph $G(S)$ of an operator $S$ is given by the usual identity:
\begin{equation*}
G(S)=\set{(h,Sh)}{h\in\dom S}.
\end{equation*}
We notice that $G(S)$ is a linear subspace of  $\hil\times \kil$ and we have 
\begin{align*}
\dom S=pr_{\hil}\langle G(S)\rangle,\qquad \ran S=pr_{\kil}\langle G(S)\rangle. 
\end{align*}

The orthocomplement $G(S)^{\perp}$ of the graph $G(S)$ plays a specific role in the following characteristic statement:
\begin{lemma}\label{L:lemma21}
Let $S$ and $T$ be linear operators  between $\hil$ and $\kil$, respectively, $\kil$ and $\hil$. If $S$ and $T$ are adjoint to each other  then we have the following identity:
\begin{align*}
G(S)^{\perp}\cap \ran M_{S,T}=W\langle G(T)\rangle.
\end{align*}
\end{lemma}
\begin{proof}
First of all $S\wedge T$  means that for each $h\in\dom S$ and $k\in \dom T$ one has 
\begin{align*}
\sip{(h,Sh)}{(-Tk,k)}=-\sip{h}{Tk}+\sip{Sh}{k}=0,
\end{align*}
whence it follows that $W\langle G(T)\rangle\subseteq G(S)^{\perp} $. On the other hand, 
\begin{eqnarray*}
M_{S,T}(0,k)=(-Tk,k)=W(k,Tk),\qquad k\in\dom T,
\end{eqnarray*}
hence $W\langle G(T)\rangle\subseteq \ran M_{S,T}.$ The reverse inclusion is obtained by the following simple argument: suppose $M_{S,T}(h,k)\in G(S)^{\perp}$ for some $h,k,$ then 
\begin{align*}
0&=\sip{(h,Sh)}{(h-Tk, Sh+k)}\\
 &=\sip{h}{h}-\sip{h}{Tk}+\sip{Sh}{k}+\sip{Sh}{Sh}\\
 &=\sip{h}{h}+\sip{Sh}{Sh}.
\end{align*}
This gives $h=0$ and therefore $M_{S,T}(0,k)=(-Tk,k)\in W\langle G(T)\rangle,$ which proves the lemma.
\end{proof}

The next result gives necessary and sufficient conditions for an operator $T$ to be 
the adjoint of an operator $S$. 
We emphasize that  conditions (ii), (iii) as natural concepts do not make use in any sense of the density of the domain   of $S$, the classical condition of the existence of $S^*$.
\begin{theorem}\label{T:theorem22}
Let $S$ and $T$ be linear operators between Hilbert spaces $\hil$ and $\kil$, respectively, $\kil$ and $\hil$. Then the following statements (i)-(iii) are equivalent:
\begin{enumerate}[\upshape (i)]
\item $S$ is densely defined and $S^*=T$.
\item $S\wedge T$ and $G(S)^{\perp}=W\langle G(T)\rangle$.
\item $S\wedge T$ and $G(S)^{\perp}\subseteq \ran M_{S,T}$.
\end{enumerate}
\end{theorem}
\begin{proof}
Under condition (i)  we clearly have 
\begin{equation*}
G(S)^{\perp}=W\langle G(S^*)\rangle=W\langle G(T)\rangle,
\end{equation*}
hence (i) implies (ii). That (ii) implies (iii) is clear. Assume finally (iii) and prove to (i). First of all we show that $S$ is densely defined. To do so let $h\in\dom S^{\perp}$. In this case, $(h,0)$ belongs to $G(S)^{\perp}$. By Lemma \ref{L:lemma21} we have $G(S)^{\perp}=W\langle G(T)\rangle$ and hence $(h,0)=(-Tk,k)$ for a certain $k\in\dom T$. Thus $k=0$ and $h=Tk=0$, as it is claimed. As a result, the adjoint operator $S^*$ exists and satisfies 
\begin{equation*}
W\langle G(S^*)\rangle=G(S)^{\perp}=W\langle G(T)\rangle.
\end{equation*}
As a consequence,  $S^*=T$, hence (iii) implies (i).
\end{proof}

\section{Operators which are adjoint of each other}
Two linear operators, say $S:\hil\to \kil$ and $T:\kil\to\hil$ are adjoint operators of each other if both of them are densely defined and the corresponding adjoint operators satisfy
\begin{equation}\label{E:S=T*,S*=T}
S^*=T\qquad \textrm{and}\qquad T^*=S.
\end{equation}
Our first result gives necessary and sufficient conditions on $S$ and $T$ in order to ensure equalities \eqref{E:S=T*,S*=T}. This  generalizes \cite[Theorem 3.4]{Popovici}:
\begin{theorem}\label{T:theorem31}
Let $S$ and $T$ be linear operators between Hilbert spaces $\hil$ and $\kil$, respectively $\kil$ and $\hil$. The following assumptions (i)-(iii) are equivalent:
\begin{enumerate}[\upshape (i)]
\item Both $S$ and $T$ are densely defined operators such that $S=T^*$ and $S^*=T$.
\item $S\wedge T$ and $\ran M_{S,T}=\hil\times\kil$.
\item $S\wedge T$ and $\ran (I+ST)=\kil$ and $\ran (I+TS)=\hil$.
\item $S\wedge T$, $G(S)^{\perp}\subseteq \ran M_{S,T}$ and $G(T)^{\perp}\subseteq \ran M_{T,S}.$
\end{enumerate}
\end{theorem}
\begin{proof} To prove that (i) implies (ii) let us recall the following well-known decomposition 
\begin{equation}\label{E:decomp}
G(S)\oplus W\langle G(S^*)\rangle=\hil\times\kil
\end{equation}
regarding the densely defined closed operator $S$ (see \cite[Theorem 4.16]{Weidmann}). Consider  $u\in\hil$ and $v\in\kil$. Since $S^*=T$ by assumption, identity \eqref{E:decomp} yields unique $h\in\dom S$ and $k\in\dom T$ such that
\begin{equation*}
(h,Sh)+(-Tk,k)=(u,v).
\end{equation*}
Consequently, $u=h-Tk$ and $v=Sh+k$, or in other words, the operator matrix $M_{S,T}$ maps $(h,k)$ into $(u,v)$:
\begin{equation*}
M_{S,T}(h,k)=(u,v).
\end{equation*}
This in turn shows that $M_{S,T}$ has  full range, i.e., $$\ran M_{S,T} =\hil\times\kil.$$ A very similar argument shows that  $$\ran M_{T,S} =\kil\times\hil,$$ as well. To see that  (ii) implies (iii) we are going to prove first that $M_{-S,-T}$  has full range too. To this aim fix $u\in\hil$ and $v\in\kil$ and choose $h\in\dom S,k\in\dom T$ such that $M_{S,T}(h,k)=(u,v).$ In other words, we have the following equalities: 
\begin{align*}
h-Tk&=u,\\
Sh+k&=v.
\end{align*}
An easy calculation shows that  $M_{-S,-T}(-h,k)=(-u,v)$, hence we have $\ran M_{-S,-T}=\hil\times\kil$, indeed. Then identity
\begin{align}\label{E:factor}
M_{S,T}\cdot M_{-S,-T}=\kismatrix{I+TS}{0}{0}{I+ST}
\end{align}
shows that $\ran (I+TS)=\hil$ and $\ran (I+ST)=\kil$ because the product of   surjective operators itself is  surjective. Implication (iii) $\Rightarrow$ (iv) goes similar: by factorization \eqref{E:factor} we see that $\ran M_{S,T}=\hil\times \kil$. Clearly, (iv) will be proved if we show that $\ran M_{T,S}=\kil\times\hil$. For let $u\in \hil$ and $v\in \kil$ and choose $h$ from $\dom S$ and $k$ from $\dom T$ such that $M_{S,T}(h,k)=(u,v)$, then an easy calculation shows that $M_{T,S}(-k,h)=(-v,u)$ whence the surjectivity of $M_{T,S}$ follows.
The missing implication (iv)$\Rightarrow$(i)   follows immediately from Theorem \ref{T:theorem22}.
\end{proof}

\begin{theorem}\label{T:kerSranT}
Let $S$ and $T$ be linear operators between Hilbert spaces $\hil$ and $\kil$, respectively, $\kil$ and $\hil$, such that they are adjoint to each other. Suppose that 
$$\ker S+\ran T=\hil \qquad \mbox{and}\qquad \ran S+\ker T=\kil.$$ 
Then both $S$ and $T$ are densely defined operators having closed range, such that they are adjoint of each other: $S^*=T$ and $T^*=S$.
\end{theorem}
\begin{proof}
First we remark that $(\ran S)^{\perp}=\ker T$. For if $z\in(\ran S)^{\perp}$ then $z=k+Sh$ for suitable $k\in\ker T$ and $h\in\dom S$, hence 
\begin{align*}
0=\sip{z}{Sh}=\sip{k+Sh}{Sh}=\sip{Tk}{h}+\sip{Sh}{Sh}=\sip{Sh}{Sh},
\end{align*}
whence $z=k\in\ker T$, i.e.,  $\ran S^{\perp}\subseteq\ker T$.
The converse inclusion is straightforward. A very similar argument shows that $(\ran T)^{\perp}=\ker S$. 

To see that $S^*=T$   we prove that $G(S)^{\perp}\subseteq\ran M_{S,T} $: for let $(v,w)\in G(S)^{\perp}$, and choose $k$ from $\dom T$ and $h$ from $\ker S$ such that $v=h+Tk$. From $(h,0)\in G(S)$ we infer that 
\begin{align*}
0=\sip{(v,w)}{(h,0)}=\sip{h}{h}+\sip{Tk}{h}=\sip{h}{h}+\sip{k}{Sh}=\sip{h}{h},
\end{align*}
whence we get $v=Tk$. On the other hand, 
\begin{align*}
0=\sip{(Tk,w)}{(x,Sx)}=\sip{Tk}{x}+\sip{w}{Sx}=\sip{k+w}{Sx}
\end{align*}
for all $x$ in $\dom S$, whence $k+w\in\ran S^{\perp}=\ker T$. This in turn implies that $w\in\dom T$ and $-Tw=Tk=v$, hence 
\begin{align*}
(v,w)=(-Tw,w)=M_{S,T}(0,w),
\end{align*}
which gives $(v,w)\in\ran M_{S,T}$. A very similar argument shows that $G(T)^{\perp}\subseteq \ran M_{T,S}$, hence $S$ and $T$ are densely defined such that $S^*=T$ and $T^*=S$. Finally, $S$ and $T$ have closed range because $(\ran S)^{\perp}=\ker T$ and $\ran S+\ker T=\kil$, and $(\ran T)^{\perp}=\ker S$ and $\ker S+\ran T=\hil$, respectively.
\end{proof}
The generalized Hellinger--Toeplitz theorem (see e.g. \cite{Riesz}) says that for everywhere defined operators $S$ and $T$ the relation $S\wedge T$ implies that both $S$ and $T$ are bounded and they are adjoint of each other. The corresponding \lq\lq dual" statement is phrased below:
\begin{theorem}
Let $S$ and $T$ be linear operators between Hilbert spaces $\hil$ and $\kil$, respectively $\kil$ and $\hil$, such that they are adjoint to each other: $S\wedge T$. Assume in addition that $S$ and $T$ are of full range, i.e. $\ran S=\kil$ and $\ran T=\hil$. Then $S$ and $T$ have (everywhere defined) bounded inverse, and $S$, $T$ are adjoint of each other.  
\end{theorem}
\begin{proof}
First of all we observe  that $\ker S=\{0\}$ (respectively, $\ker T=\{0\}$), hence 
the inverse operators $S^{-1}$ and $T^{-1}$ exist. Indeed, if $Sh=0$ for some $h\in\dom S$ then for every $k$ from $\dom T$ one has
\begin{equation*}
0=\sip{Sh}{k}=\sip{h}{Tk},
\end{equation*}
whence $h=0$, being orthogonal to $\ran T=\hil$. Here, $S^{-1}$ and $T^{-1}$ are adjoint to each other because for $h\in\hil$ and $k=Sh'\in\kil$ we have
\begin{align*}
\sip{T^{-1}h}{k}=\sip{T^{-1}h}{Sh'}=\sip{TT^{-1}h}{h'}=\sip{h}{h'}=\sip{h}{S^{-1}k}.
\end{align*}
By the remark preceding the theorem, $S^{-1}$ and $T^{-1}$ are bounded operators such that $(S^{-1})^*=T^{-1}$. Consequently, $S$ and $T$ are densely defined and fulfill \eqref{E:S=T*,S*=T}.
\end{proof}
\section{Adjoint of sums and products}
Given two densely defined linear operators $R,S$ acting between Hil\-bert spaces $\hil$ and $\kil$ one cannot expect in general the additive identity 
\begin{equation}\label{E:(R+S)^*}
(R+S)^*=R^*+S^*,
\end{equation}
all the more so, because the operator on the left side does not exist in general. A well-known assumption  for \eqref{E:(R+S)^*} is that any of the operators be bounded (see \cite{Riesz}). For more general results the reader may consult   \cite{Birman,Gesztesy, KATO, PopSZTZS,SZTZS2015, Weidmann}. In the next theorem we provide necessary and sufficient conditions in order that \eqref{E:(R+S)^*} be satisfied (cf. also \cite[Theorem 2.1]{SZTZS2015}).
\begin{theorem}\label{T:R+S}
Let $R, S$ be densely defined linear operators acting between $\hil$ and $\kil$. Then the following statements are equivalent:
\begin{enumerate}[\upshape (i)]
 \item $R+S$ is densely defined and $(R+S)^*=R^*+S^*$.
 \item $G(R+S)^{\perp}=W\langle G(R^*+S^*)\rangle$.
 \item $G(R+S)^{\perp}\subseteq \ran (M_{R+S,R^*+S^*})$.
\end{enumerate}
\end{theorem}
\begin{proof}
A direct calculation shows that $R+S$ and $R^*+S^*$ are adjoint to each other, hence the statement of the theorem follows from Theorem \ref{T:theorem22}.
\end{proof}
Next we concern with the multiplicative version
\begin{equation}\label{E:(RS)*}
 (RS)^*=S^*R^*,
\end{equation}
where $R,S$ are operators acting between Hilbert spaces $\hil_2$ and $\hil_3$, respectively, $\hil_1$ and $\hil_2$. Just like in the ``additive" case, we may not expect \eqref{E:(RS)*} to hold in general. The multiplicative identity can be guaranteed by some strongly restrictive conditions, for example, if $R$ is (everywhere defined) bounded operator (see \cite{Riesz}) or when $S$ admits bounded inverse (see \cite{Birman}). As an application of Theorem \ref{T:theorem22} we gain necessary and sufficient conditions (cf. also \cite{SZTZS2015}):
\begin{theorem}\label{T:RS*}
Let $R,S$ be operators acting between Hilbert spaces $\hil_2$ and $\hil_3$, respectively $\hil_1$ and $\hil_2$. Then the following assertions are equivalent:
\begin{enumerate}[\upshape (i)]
\item $RS$ is densely defined and $(RS)^*=S^*R^*$.
\item $G(RS)^{\perp}=W\langle G(S^*R^*)\rangle$.
 \item $G(RS)^{\perp}\subseteq \ran (M_{RS,S^*R^*})$.
\end{enumerate}
\end{theorem}
\begin{proof}
A direct calculation shows that $RS$ and $S^*R^*$ are adjoint to each other, hence the statement of the theorem follows from Theorem \ref{T:theorem22}.
\end{proof}

Closures of  sum and product of two linear operators may appear as the sum, respectively the product of the closures of the operators (see also Appendix B of \cite{Gesztesy}):
\begin{corollary}
Let $R,S$ be densely defined closable operators on a Hilbert space $\hil$ such that the following two relations are satisfied: 
\begin{enumerate}[\upshape a)]
 \item $G(R+S)^{\perp}\subseteq \ran(M_{R+S,R^*+S^*})$,
 \item $G(R^*+S^*)^{\perp}\subseteq\ran(M_{R^*+S^*,R^{**}+S^{**}})$.
\end{enumerate}
Then we have the next additive property of the closure operation:
\begin{align*}
(R+S)^{**}=R^{**}+S^{**}.
\end{align*}
\end{corollary}
\begin{proof}
Assumption a) implies by Theorem \ref{T:R+S} that $(R+S)^*=R^*+S^*$. One more application of the same theorem proves the statement.
\end{proof}
\begin{corollary}
Given two closable densely defined operators $R$ and $S$ on a Hilbert space $\hil$ such that they satisfy the following two conditions:
\begin{enumerate}[\upshape a)]
 \item $G(RS)^{\perp}\subseteq \ran (M_{RS,S^*R^*})$,
 \item $G(S^*R^*)^{\perp}\subseteq \ran(M_{S^*R^*,R^{**}S^{**}})$.
\end{enumerate}
Then we have the next multiplicative property of the closure operation:
\begin{align*}
(RS)^{**}=R^{**}S^{**}.
\end{align*}
\end{corollary}
\begin{proof}
Assumption  a) implies $(RS)^*=S^*R^*$ in view of Theorem \ref{T:RS*}. Hence one more use of the theorem yields the statement.
\end{proof}
\begin{corollary}
Let $R,S$ be linear operators in the Hilbert space $\hil$ and assume they are adjoint to themselves (that is,  $R\wedge R$ and $S\wedge S$). Then the following two assertions are equivalent:
\begin{enumerate}[\upshape (i)]
 \item $\dom(RS)$ is dense and $(RS)^*=SR$,
 \item $G(RS)^{\perp}\subseteq \ran(M_{RS,SR})$.
\end{enumerate}
\end{corollary}
\begin{proof}
We need only to check that $RS\wedge SR$: 
\begin{align*}
\sip{RSh}{k}=\sip{Sh}{Rk}=\sip{h}{SRk},\qquad h\in\dom(RS), k\in\dom(SR),
\end{align*}
indeed.
\end{proof}
A symmetric version of the above result reads as follows:
\begin{corollary}
Given two linear operators $R,S$ in the Hilbert space $\hil$ satisfying $R\wedge R$ and $S\wedge S$. The following two assertions  are equivalent:
\begin{enumerate}[\upshape (i)]
 \item $RS$ and $SR$ are adjoint of each other (i.e., $\dom (RS)$ and $\dom (SR)$ are dense and $(RS)^*=SR$, $(SR)^*=RS$),
 \item $G(RS)^{\perp}\subseteq \ran (M_{RS,SR})$ and $G(SR)^{\perp}\subseteq \ran(M_{SR,RS}).$
\end{enumerate}
\end{corollary}
\begin{proof}
This is an straightforward consequence of the preceding corollary.
\end{proof}

\section{Range of adjoint operators}
For a densely defined closed linear operator $S$ between $\hil$ and $\kil$ one has the orthogonal decomposition
\begin{align*}
\hil = \ker S \oplus \overline{\ran S^*},
\end{align*}
so that the elements of the range closure of the adjoint operator $S^*$ are obtained as the ones being orthogonal to the kernel of $S$. Describing the elements of the range of $S^*$ is more involved. A metric characterization of $\ran S^*$ is given in \cite{Sebestyen83a} by the first named author. Below we provide a generalization of that result for the case of operators which are adjoint to each other.  
\begin{theorem}\label{T:Sebestyen83}
Let $\hil$, $\kil$ be real or complex Hilbert spaces and let $S:\hil\to\kil$ and $T:\kil\to\hil$ be linear operators adjoint to each other and assume  that $\overline{\ran S}\cap pr_{\kil}\langle W\langle G(S)^{\perp}\rangle\rangle\subseteq \dom T$. For a given $z\in\hil$ the following assertions are equivalent:
\begin{enumerate}[\upshape (i)]
 \item $z\in\ran T+(\dom S)^{\perp}$;
 \item There is $M_z\geq0$ such that 
 \begin{align*}
    \abs{\sip{x}{z}}\leq M_z\|Sx\|,\qquad \forall x\in\dom S. 
\end{align*}
\end{enumerate}
\end{theorem}
\begin{proof}
Assume first (i) and choose $u\in\dom T$ and $v\in(\dom S)^{\perp}$ such that $z=Tu+v$. For every $x$ from $\dom S$ we have
\begin{align*}
\abs{\sip{x}{z}}=\abs{\sip{x}{Tu+v}}=\abs{\sip{Sx}{u}} \leq \|u\| \|Sx\|,
\end{align*}
which implies (ii). For the converse implication observe that (ii) forces the following linear functional
\begin{align*}
\ran S\to\mathbb{K};\qquad Sx\mapsto \sip{x}{z}
\end{align*}
to be continuous. By the Riesz representation theorem, there is a unique vector $u\in\overline{\ran S}$ such that 
\begin{align*}
\sip{x}{z}=\sip{Sx}{u},\qquad \forall x\in\dom S.
\end{align*}
Consequently,
\begin{align*}
\sip[\bigg]{\pair{-Sx}{x}}{\pair{u}{z}}=\sip{x}{z}-\sip{Sx}{u}=0,\qquad \forall x\in\dom S,
\end{align*}
which yields $\kispair{u}{z}\in  W\langle G(S)\rangle^{\perp}$. Hence 
\begin{align*}
u\in \overline{\ran S}\cap pr_{\kil}\langle W\langle G(S)^{\perp}\rangle\rangle\subseteq \dom T,
\end{align*}
and therefore 
\begin{equation}\label{E:z-tu}
\sip{x}{z}=\sip{Sx}{z}=\sip{x}{Tu},\qquad \forall x\in\dom S, 
\end{equation}
because $T\wedge S$. Note that \eqref{E:z-tu} simultaneously yields $z-Tu\in(\dom S)^{\perp}$. Consequently, 
\begin{equation*}
z=Tu+(z-Tu)\in \ran T+(\dom S)^{\perp},
\end{equation*}
which completes the proof.
\end{proof}
Since the adjoint of a densely defined linear operator $S$ fulfills 
\begin{equation*}
G(S^*)= W\langle G(S)^{\perp}\rangle,
\end{equation*}
 we just obtain identity 
\begin{equation}\label{E:prkw}
pr_{\kil}\langle W\langle G(S)^{\perp}\rangle\rangle=\dom S^*.
\end{equation}
 As a consequence we retrieve Theorem 1 of \cite{Sebestyen83a} which characterizes the range of the adjoint operator:
\begin{corollary}\label{C:Sebestyen83}
Let $S$ be a  densely defined linear operator between $\hil$ and $\kil$. For $z\in\hil$  the following conditions are equivalent
\begin{enumerate}[\upshape (i)]
 \item $z\in\ran S^*$;
 \item There is $M_z\geq0$ such that 
  \begin{equation*}
     \abs{\sip{x}{z}}\leq M_z\|Sx\|,\qquad \forall x\in\dom S.
  \end{equation*}
\end{enumerate}
\end{corollary}
\begin{proof}
The preceding theorem applies with $T=S^*$ because of $(\dom S)^{\perp}=\{0\}$ identity \eqref{E:prkw}.  
\end{proof}
\begin{corollary}
Let $T$ be a densely defined closed linear operator between $\hil$ and $\kil$. For $z\in\kil$ the following conditions are equivalent:
\begin{enumerate}[\upshape (i)]
 \item $z\in\ran T$;
 \item There is $M_z\geq0$ such that 
  \begin{equation*}
     \abs{\sip{x}{z}}\leq M_z\|T^*x\|,\qquad \forall x\in\dom T^*.
  \end{equation*}
\end{enumerate}
\end{corollary}
\begin{proof}
The proof is immediate from Corollary \ref{C:Sebestyen83} because $T=T^{**}$ by assumptions.
\end{proof}
From the Banach closed range theorem it is known that the adjoint $S^*$ of a densely defined closed operator $S$ is of full range if and only if the operator is bounded from below. The next theorem establishes a generalization of that fact for operators which are adjoint to each other.
\begin{theorem}
Let $\hil$, $\kil$ be real or complex Hilbert spaces and let $S:\hil\to\kil$ and $T:\kil\to\hil$ be linear operators which are adjoint to each other. Assume furthermore that $\overline{\ran S}\cap pr_{\kil}\langle W\langle G(S)^{\perp}\rangle\rangle\subseteq \dom T$. Then  the following assertions are equivalent:
\begin{enumerate}[\upshape (i)]
 \item There is $c>0$ such that $\|Sx\|\geq c\|x\|$ for all $x\in\dom S$.
 \item $\ran T+(\dom S)^{\perp}=\hil$.
\end{enumerate}
\end{theorem}
\begin{proof}
Assume first (i) and consider $z\in\hil.$ For any $x\in \dom S$  we have 
\begin{align*}
\abs{\sip{x}{z}}\leq \|x\|\|z\|\leq c^{-1}\|y\|\|Sx\|,
\end{align*}
hence $z\in\ran T+(\dom S)^{\perp}$, according to Theorem \ref{T:Sebestyen83}. Suppose conversely (ii). Our first claim is to check that  $S$ is one-to-one: for if $x\in\ker S$, then for any $y\in\dom T$ and $u\in(\dom S)^{\perp}$ we have 
\begin{align*}
\sip{x}{Ty+u}=\sip{x}{Ty}=\sip{Sx}{y}=0,
\end{align*}
hence  $x=0$, indeed. 
The inverse $S^{-1}$ of $S$ exists therefore as an operator $\kil\supseteq \ran S\to \hil$. Furthermore, for any $z\in\dom S^{-1}$, $y\in\dom T$ and $u\in(\dom S)^{\perp}$ we have 
\begin{align*}
\abs{\sip{Ty+u}{S^{-1}z}}=\abs{\sip{Ty}{S^{-1}z}}=\abs{\sip{y}{z}}\leq \|z\|\|y\|.
\end{align*}
This in turn shows that the set $\set{S^{-1}z}{z\in\dom S^{-1}, \|z\|\leq 1}$ is weakly bounded and hence also uniformly bounded according to the Banach uniform boundedness principle. That means that there exists $M> 0$ such that $\|S^{-1}z\|\leq M\|z\|$ for all $z\in\dom T^{-1}$, which clear\-ly implies (i).  
\end{proof}
\begin{corollary}
For a densely defined linear operator $S$ between $\hil$ and $\kil$ the following statements are equivalent:
\begin{enumerate}[\upshape (i)]
 \item There is $c>0$ such that $\|Sx\|\geq c\|x\|$ for all $x\in\dom S$. 
 \item $S^*$ is a full range operator, i.e., $\ran S^*=\hil$.
\end{enumerate}
\end{corollary}
\begin{proof}
This is an immediate consequence of the preceding theorem and  \eqref{E:prkw}.
\end{proof}

\section{Skew-adjoint and selfadjoint operators}
An operator $S$ acting on a real or complex Hilbert space $\hil$ is called symmetric if it is adjoint to itself, i.e., $S$ satisfies  $S\wedge S$:  
\begin{equation*}
\sip{Sh}{h'}=\sip{h}{Sh'}, \qquad  h,h'\in\dom S.
\end{equation*}
A linear operator $S$ in $\hil$ is said to be skew-symmetric if $S\wedge (-S)$, that is 
\begin{align*}
\sip{Sh}{h'}=-\sip{h}{Sh'}, \qquad  h,h'\in\dom S.
\end{align*}
We say that a densely defined operator $S$ is selfadjoint if $S^*=S$ and we call $S$ skew-adjoint  if $S^*=-S$. If the underlying Hilbert space is complex then the mapping $S\mapsto iS$ establishes a bijective correspondence between symmetric and skew-symmetric, and also between selfadjoint and skew-adjoint operators. 

In our first result we are going to provide a characterization of skew-adjoint operators among skew-symmetric ones. This simultaneously extends \cite[Theorem 4.1]{Characterization} and \cite[Theorem 2.1]{Squareselfadj}. We emphasize again that no assumption regarding the density of the domain of $S$ is required and also  the underlying space is allowed to be either real or complex. 
\begin{theorem}\label{T:theorem51}
Let $S$ be a skew-symmetric operator in the Hilbert space $\hil$. The following assertions (i)-(vi) are equivalent:
\begin{enumerate}[\upshape (i)]
\item $S$ is skew-adjoint.
\item $\ran M_{S,-S}=\hil\times\hil$.
\item $G(S)^{\perp}\subseteq \ran M_{S,-S}.$
\item $\ran (I-S^2)=\hil.$
\item $\ran(I\pm S)=\hil.$
\item $S^2$ is densely defined and selfadjoint: $(S^2)^*=S^2$.
\end{enumerate}
\end{theorem}
\begin{proof}
The equivalence of (i), (ii) and (iii) is a simple consequence of Theorem \ref{T:theorem31} (choose $T=-S$). Observe on the other hand that
\begin{align*}
M_{S,-S}\cdot M_{-S,S}=M_{-S,S}\cdot M_{S,-S}={\left(\begin{array}{cc}\!\! I-S^2 &  0\!\!\\ \!\! 0& I-S^2\!\!\end{array}\right)}
\end{align*}
due to skew-symmetry, hence assertions (ii) and (iv) are also equivalent. The equivalence between (iv) and (v) is due to formula
\begin{equation*}
(I+S)(I-S)=I-S^2=(I-S)(I+S).
\end{equation*}
Finally, the (not necessarily densely defined) positive symmetric operator $A:=-S^2$ is selfadjoint if and only if $\ran (I+A)=\hil$ (see e.g. \cite[Proposition 3.1]{Characterization}), which proves the equivalence of (iv) and (vi).
\end{proof}

A classical result of von Neumann \cite{vonNeumann1930} says that a densely defined closed symmetric operator $S$ on a complex Hilbert space $\hil$ is selfadjoint (i.e., $S=S^*$) if and only if $\ran(iI\pm S)=\hil$. In the next result we are going to improve this statement: 
\begin{theorem}
Let $\hil$ be a complex Hilbert space and $S$ be a linear operator in $\hil$. Then the following statements are equivalent: 
\begin{enumerate}[\upshape (i)]
\item $S$ is selfadjoint (i.e., $S$ is densely defined and $S^*=S$).
\item $S$ is symmetric and $\ran(iI\pm S)=\hil$.
\end{enumerate}
\end{theorem}
\begin{proof}
The theorem follows immediately from Theorem \ref{T:theorem51} by observing that $S$ is selfadjoint if and only if $iS$ is skew-adjoint and similarly, $S$ is symmetric if and only if $iS$ is skew-symmetric.
\end{proof}
On full generality we have the following result (cf. also \cite[Corollary 3.6]{Popovici}, \cite[Theorem 5.1]{Characterization} and \cite[Theorem 2.2]{Squareselfadj}):
\begin{theorem}\label{T:theorem53}
For a symmetric operator $S$ in a real or complex Hilbert space $\hil$ the following assertions are equivalent:
\begin{enumerate}[\upshape (i)]
\item $S$ is selfadjoint.
\item $\ran M_{S,S}=\hil\times\hil$.
\item $G(S)^{\perp}\subseteq \ran M_{S,S}.$
\item $\ran (I+S^2)=\hil.$
\item $S^2$ is densely defined and selfadjoint: $(S^2)^*=S^2$.
\end{enumerate}
\end{theorem}
\begin{proof}
The proof of the statement is very similar to that of Theorem \ref{T:theorem51}: equivalence of (i), (ii) and (iii) follows immediately from Theorem \ref{T:theorem31}. From formula 
\begin{align*}
M_{S,S}\cdot M_{-S,-S}= M_{-S,-S}\cdot M_{S,S}={\left(\begin{array}{cc}\!\! I+S^2 &  0\!\!\\ \!\! 0& I+S^2\!\!\end{array}\right)}
\end{align*}
we just conclude the equivalence of (ii) and (iv). Finally, the missing equivalence (iv)$\Leftrightarrow$(v) follows from the observation  that $S^2$ is a positive symmetric operator.
\end{proof}

As an immediate, and somewhat surprising consequence we obtain that any symmetric square root of a positive selfadjoint operator itself is selfadjoint:
\begin{corollary}
Let $A$ be a positive selfadjoint operator in the real or complex Hilbert space $\hil$. If $B$ is a symmetric operator such that $B^2=A$ then $B$ is selfadjoint.
\end{corollary}
\begin{proof}
Straightforward from Theorem \ref{T:theorem53}.
\end{proof}
A characterization of closed range selfadjoint operators is established in the next result:
\begin{corollary}
Let $S$ be a (not necessarily densely defined or closed) symmetric operator in the real or complex Hilbert space $\hil$ such that  $$\ker S+\ran S=\hil.$$ Then $S$ is a (densely defined and) selfadjoint operator with closed range.
\end{corollary}
\begin{proof}
First we remark that $\ran S^{\perp}=\ker S$. Indeed, for if $z\in\ran S^{\perp}$ then $z=k+Sh$ for suitable $k\in\ker S$ and $h\in\dom S$, so
\begin{align*}
0=\sip{z}{Sh}=\sip{k+Sh}{Sh}=\sip{Sk}{h}+\sip{Sh}{Sh}=\sip{Sh}{Sh},
\end{align*}
whence it follows that $z=k\in\ker S$, i.e.,  $\ran S^{\perp}\subseteq\ker S$.
The converse inclusion is straightforward. To see that $S$ is selfadjoint we prove that $G(S)^{\perp}\subseteq\ran M_{S,S}$: for let $(v,w)\in G(S)^{\perp}$. Choose $h$ from $\dom S$ and $k$ from $\ker S$ such that $v=k+Sh$. Then, as $(k,0)\in G(S)$, we infer that 
\begin{align*}
0=\sip{(v,w)}{(k,0)}=\sip{Sh}{k}+\sip{k}{k}=\sip{k}{k},
\end{align*}
whence we get $v=Sh$. On the other hand, 
\begin{align*}
0=\sip{(Sh,w)}{(x,Sx)}=\sip{Sh}{x}+\sip{w}{Sx}=\sip{h+w}{Sx}
\end{align*}
for all $x$ in $\dom S$, whence $h+w\in\ran S^{\perp}=\ker S$. Hence $w\in\dom S$ and $-Sw=Sh=v$. Therefore
\begin{align*}
(v,w)=(-Sw,w)=M_{S,S}(0,w),
\end{align*}
which gives $(v,w)\in\ran M_{S,S}$. An application of Theorem \ref{T:theorem53} implies that $S$ is selfadjoint. 
\end{proof}
As a straightforward consequence we also receive a generalization of \cite[Exercise 10.4]{Weidmann} by Weidmann:
\begin{corollary}
Let $S$ be a symmetric operator in a real or complex Hilbert space $\hil$ such that $\ker(S+\lambda I)+\ran(S+\lambda I)=\hil$ for some real $\lambda$. Then $S$ is densely defined and  selfadjoint.
\end{corollary}
We close the section with a ``dual" version of the classical Hellinger--Toeplitz theorem that improves  \cite[Theorem 2.9]{Stone} of M. H. Stone: 
\begin{corollary}\label{C:corollary67}
A symmetric operator with full range is automatically den\-sely defined and selfadjoint possessing a bounded inverse.
\end{corollary}
\begin{proof}
A surjective  symmetric operator $S$ obviously fulfills the conditions of the preceding corollary, hence $S$ must be selfadjoint. Since $S$ has trivial kernel, $S^{-1}$ exists as an everywhere defined selfadjoint, hence closed operator. The closed graph theorem forces $S^{-1}$ to be bounded. 
\end{proof}
\section{Positive selfadjoint operators}
A not necessarily densely defined linear operator $A$ acting in a real or complex Hilbert space $\hil$ is called positive if it satisfies
\begin{equation*}
\sip{Ah}{h}\geq0,\qquad \mbox{for all $h\in\dom A$.}
\end{equation*}
A positive operator in a complex Hilbert space is automatically symmetric but this is not the case on real Hilbert spaces. To begin with we offer a  characterization of positive selfadjoint operators (see also \cite[Proposition 3.1]{Characterization}). 
\begin{theorem}\label{T:theorem55}
Let $A$ be a positive linear operator in a real or complex Hilbert space $\hil$. Then the following assertions are equivalent:
\begin{enumerate}[\upshape (i)]
 \item $A$ is (densely defined and) selfadjoint.
 \item $A$ is symmetric (i.e., $A\wedge A$) and $\ran (I+A)=\hil$.
\end{enumerate}
\end{theorem}
\begin{proof}
If $A$ is positive and selfadjoint then $I+A$ is a  closed operator that is bounded from below, hence it has closed range. By selfadjointness this implies $\ran (I+A)=\hil.$ Conversely, $I+A$ is a full range symmetric, and hence selfadjoint operator in view of Corollary \ref{C:corollary67}. As a result, $A=(I+A)-I$ is also selfadjoint. 
\end{proof}

If $T$ is a densely defined  closed operator between $\hil$ and $\kil$ then $T^*T$ and $TT^*$ are both selfadjoint operators in the Hilbert spaces $\hil$ and $\kil$, respectively. It is also known that the domain of the corresponding positive selfadjoint square roots $\abs T:=(T^*T)^{1/2}$ and $\abs{T^*}=(TT^*)^{1/2}$ satisfy
\begin{align*}
\dom \abs T=\dom T,\qquad \dom \abs{T^*}=\dom T^*.
\end{align*}
As  an application of Theorem \ref{T:theorem53} we prove below an inequality that was called ``mixed Schwarz'' inequality by P. R. Halmos  in the bounded operator case (see \cite[Problem 138]{Halmos}): 
\begin{theorem}
Let $T$ be a densely defined closed linear operator between the Hilbert spaces $\hil$ and $\kil$. Then 
\begin{align*}
\abs{\sip{Tx}{y}}^2\leq \sip{\abs{T}x}{x}\cdot\sip{\abs{T^*}y}{y}, 
\end{align*}
for $x\in\dom T$ and $y\in\dom T^*$.
\end{theorem}
\begin{proof}
Consider the polar decomposition $T=U\abs{T}$ of $T$. Then $T^*=\abs T U^*$ and hence 
\begin{align*}
\abs T^2= T^*T=\abs TU^*UT.
\end{align*}
Consequently, 
\begin{align*}
TT^*=U\abs T^2 U^*=UT^*TU^*=(U\abs T U^*)^2,
\end{align*}
i.e., $U\abs T U^*$ is a positive symmetric operator with selfadjoint square. By Theorem \ref{T:theorem53}, $U\abs T U^*$  itself  is  selfadjoint   and $U\abs T U^*=\abs{T^*}$, accordingly. For $x\in\dom T$ and  $y\in\dom T^*$ we have therefore $U^*y\in \dom \abs T$ and 
\begin{align*}
\abs{\sip{Tx}{y}}^2&=\abs{\sip{\abs Tx}{U^*y}}^2\\
                    &\leq \sip{\abs Tx}{x}\cdot \sip{\abs{T}U^*y}{U^*y}\\
                    &=\sip{\abs Tx}{x}\cdot \sip{\abs{T^*} y}{ y},
\end{align*}
completing the proof.
\end{proof}
As it was shown in \cite{TadjointT},  $T^*T$ is not necessarily selfadjoint but it always has a positive selfadjoint extension, namely its smallest, so called Krein-von Neumann extension. Below we give an alternative proof of that statement, constructing a selfadjoint extension of $T$ by means of the canonical graph  projection of $\hil\times \kil$ onto $\overline{G(T)}$. 
\begin{theorem}\label{T:T*Text}
Let $T$ be a densely defined  linear operator acting between $\hil$ and $\kil$. Denote by $P$  the  orthogonal projection of $\hil\times \kil$ onto $\overline{G(T)}$ and set
\begin{equation}
 G_S:=P\langle\hil\times\{0\}\rangle.
\end{equation}
Then $G_S$ is the graph of a densely defined closable operator $S$ which fulfills the following properties:
\begin{enumerate}[\upshape a)]
 \item $G_S\subseteq \overline{G(T)}$;
 \item $\ran S\subseteq \dom T^*$;
 \item $T^*S$ is a positive selfadjoint extension of $T^*T$.
\end{enumerate}
Furthermore, $S$ is the unique linear operator possessing all the properties a)-c).  
\end{theorem}
\begin{proof}
Observe first that $G_S$ obviously  satisfies a). To prove b) take an element  $\kispair{x}{y}=P\kispair{u}{0}$ from $G_S$ and choose $z\in\dom T^*$ such that  $(I-P)\kispair{u}{0}=\kispair{-T^*z}{z}$ according to identity $G(T)^{\perp}=W\langle G(T^*)\rangle$. Then one obtains  
\begin{equation*}
\pair{u}{0}=\pair{x-T^*z}{y+z},
\end{equation*}
which yields $y=-z\in\dom T^*$.  Consequently, $\ran G_S\subseteq \overline{G(T)}$. Next we prove that $\overline{G_S}$ is the graph of an operator: with this aim  let $(0,w)\in \overline{G_S}$. Then $w$ belongs to $\overline{\dom T^*}$ because of the preceding observation. Furthermore, $(0,w)\in \overline{G(T)}$ implies
\begin{align*}
0=\sip[\bigg]{\pair{0}{w}}{\pair{-T^*z}{z}}=\sip{w}{z}
\end{align*}
for all $z$ from $\dom T^*$,  hence  $w\in(\dom T^*)^{\perp}$ and therefore  $w=0$. We see now that $G_S$ is the graph of a closable operator, say $S$. Our next claim is to show that $S$ is densely defined. To this end consider a vector $u$ from $(\dom S)^{\perp}$ and observe that 
\begin{equation*}
0=\sip[\bigg]{\pair{u}{0}}{\pair{x}{Sx}}
\end{equation*}
holds for all $x$ from $\dom S$. Hence $\kispair{u}{0}\in G(S)^{\perp}$ and consequently, 
\begin{align*}
 \left\|P\pair{u}{0}\right\|^2=\sip[\bigg]{P\pair{u}{0}}{\pair{u}{0}}=0.
\end{align*}
This means that $\kispair{u}{0}$ belongs to $G(T)^{\perp}$  and thus $\kispair{u}{0}=\kispair{-T^*z}{z}$ for certain $z$ from $\dom T^*$. This in turn shows that   $z=0$ and therefore $u=-T^*z=0$, as it is claimed. 

Now we see that $S$ is a densely defined closable operator such that $G(S)\subseteq \overline{G(T)}$. In particular, $S^*S$ is a positive symmetric extension of $T^*S$ and thus $T^*S$ itself is positive and symmetric. According to Theorem \ref{T:theorem55}, in order to prove  selfadjointness our only duty is to show that $I+T^*S$ has full range. To this aim consider  $u$ from $\hil$,  then
\begin{equation*}
\pair{u}{0}=P\pair{u}{0}+(I-P)\pair{u}{0}=\pair{x-T^*z}{Sx+z}
\end{equation*}
for certain $x\in\dom S$ and $z\in\dom T^*$. That gives $Sx=-z\in\dom T^*$ and therefore   
\begin{align*}
u=x-T^*z=x+T^*Sx\in\ran (I+T^*S),
\end{align*}
as it is claimed. 

In order to see that $T^*S$ extends $T^*T$ let us consider $v$ in $\dom T^*T$, then   
\begin{align*}
\pair{v+T^*Tv}{0}=\pair{v}{Tv}+\pair{-T^*(-Tv)}{-Tv}\in \overline{G(T)}+G(T)^{\perp},
\end{align*}
whence it follows that 
\begin{align*}
\pair{v}{Tv}=P\pair{v+T^*Tv}{0}\in G(S).
\end{align*} 
In particular, $v\in\dom S$ and $Sv=Tv$. Hence, $T^*Sv=T^*Tv$.

To the uniqueness part of the statement consider an operator $R$ which satisfies conditions a)-c),  with $S$ replaced by $R$. By b),   $Rx\in\dom T^*$  for every $x\in\dom R$. Hence $\kispair{T^*Rx}{-Rx}\in W\langle G(T^*)\rangle$ and $\kispair{x}{Rx}\in \overline{G(T)}$ by part a) and therefore 
\begin{equation*}
\pair{x+T^*Rx}{0}=\pair{x}{Rx}+\pair{-T^*(-Rx)}{-Rx}\in \overline{G(T)}+ G(T)^{\perp},
\end{equation*}
whence we infer that $\kispair{x}{Rx}\in G(S)$ and $R\subseteq S$, accordingly. For the reverse inclusion consider $u$ from $\hil$. By positive selfadjointness there is a unique $x$ in $\dom T^*R$ such that $u=x+T^*Rx$. Just as above, 
\begin{equation*}
\pair{u}{0}=\pair{x}{Rx}+\pair{T^*Rx}{-Rx}\in \overline{G(T)}+G(T)^{\perp},
\end{equation*}
whence we get $P\kispair{u}{0}=\kispair{x}{Rx}\in G(R).$ Consequently, 
\begin{align*}
G(S)=P\langle \hil\times \{0\}\rangle \subseteq G(R),
\end{align*}
and therefore $S\subseteq R$. The proof is  complete. 
\end{proof}
We remark that $S=T^{**}$ for closable $T$ may only happen if $T$ is bounded. Indeed, in this  case  we have $\ran T^{**}\subseteq \dom T^*$ and hence $T$ is bounded according to \cite[Lemma 2.1]{TZS:closedrange}. For closed $T$  we can establish the following result:
\begin{corollary}
Let $T$ be a densely defined   closed linear operator acting between Hilbert spaces $\hil$ and $\kil$. Then the only linear operator $S$ fulfilling a)-c) in the previous theorem is the restriction of $T$ to $\dom T^*T$. In particular we have 
\begin{align*}
G(T|_{\dom T^*T})=P\langle \hil\times\{0\}\rangle.
\end{align*}
\end{corollary}
\begin{proof}
If $T$ is densely defined and closed then $R:=T|_{\dom T^*T}$ fulfills any of properties a)-c) of the previous theorem. The uniqueness part of the result leads us to the desired conclusion.  
\end{proof}

\begin{proposition}
With  notation of Theorem \ref{T:T*Text} we  have 
\begin{equation}\label{E:domS}
 \dom S=\set{x\in\hil}{\exists y\in\dom T^*, (x,y)\in \overline{G(T)}}
\end{equation}
\end{proposition}
\begin{proof}
Consider first $x\in \dom S$, then $y=Sx$ belongs to $\dom S^*$ due to Theorem \ref{T:T*Text} b), and also $(x,Sx)\in G(S)\subseteq \overline{G(T)}$, according to Theorem \ref{T:T*Text} a). This proves inclusion ``$\subseteq $''. For the reverse inclusion let $x\in \hil$ and $y\in \dom T^*$ such that $(x,y)\in \overline{G(T)}$. Then 
\begin{align*}
\pair{x+T^*y}{0}=\pair{x}{y}+\pair{T^*y}{-y}\in \overline{G(T)}+G(T)^{\perp},
\end{align*}
whence we get $\kispair{x}{y}=P\kispair{x+T^*y}{0}\in G(S)$ and $x\in\dom S$, accordingly.
\end{proof}
Thanks to formula \eqref{E:domS} describing the domain of $S$ we are able to specify also the domain of $T^*S$:
\begin{corollary}
With notation of Theorem \ref{T:T*Text}, the domain of the positive selfadjoint operator $T^*S$ is given by 
\begin{align*}
\dom T^*S=\set{x\in \hil}{\exists y\in\dom T^*, \exists \seq{x} \textrm{~in $\dom T$}, x_n\to x, Tx_n\to y}
\end{align*}
\end{corollary}
\begin{proof}
It follows from Theorem \ref{T:T*Text} b) that $\dom S=\dom T^*S$. The assertion is therefore an immediate consequence of \eqref{E:domS}.
\end{proof}
As $T^*S$ is a positive selfadjoint operator in $\hil$, by virtue of Theorem \ref{T:theorem55}, $I+T^*S$  has an everywhere defined bounded inverse. Below we are going to analyze  $(I+T^*S)^{-1}$:
\begin{proposition}\label{P:I+TSquadratic}
With notation of Theorem \ref{T:T*Text}, the quadratic form of $(I+T^*S)^{-1}$ is calculated by
\begin{align*}
\sip{(I+T^*S)^{-1}x}{x}=\left\|P\kispair{x}{0}\right\|^2,\qquad x\in\hil.
\end{align*}
\end{proposition}
\begin{proof}
Denote by $A$ the inverse of $I+T^*S$. For  $x\in H$ we have $x=(I+T^*S)Ax$, hence
\begin{align*}
\pair{x}{0}=\pair{Ax}{SAx}+\pair{T^*SAx}{-SAx}\in \overline{G(T)}+G(T)^{\perp},
\end{align*}
which gives $P\kispair{x}{0}=\kispair{Ax}{SAx}$. Consequently, 
\begin{align*}
\sip{Ax}{x}=\sip[\bigg]{\pair{Ax}{SAx}}{\pair{x}{0}}=\sip[\bigg]{P\pair{x}{0}}{\pair{x}{0}}=\left\|P\pair{x}{0}\right\|^2, 
\end{align*}
which completes the proof.
\end{proof}
\begin{corollary}
With notation of Theorem \ref{T:T*Text} we have 
\begin{eqnarray}\label{E:I+TSinverz}
(I+T^*S)^{-1}=pr_{\hil}Ppr_{\hil}^*.
\end{eqnarray}
\end{corollary}
\begin{proof}
An immediate calculation shows that the adjoint operator $pr_{\hil}^*:\hil\to\hil\times \kil$ acts by
\begin{align*}
pr_{\hil}^*x=\pair x0,\qquad x\in\hil.
\end{align*}
Thus by Proposition \ref{P:I+TSquadratic} we get that 
\begin{align*}
\sip{(I+T^*S)^{-1}x}{x}=\sip{pr_{\hil}Ppr_{\hil}^*x}{x},\qquad x\in\hil,
\end{align*}
which gives equality \eqref{E:I+TSinverz}.
\end{proof}
\section{Closed operators}
A linear operator $T$ between two Hilbert spaces $\kil$ and $\hil$ is closed if its graph $G(T)$ is a closed linear subspace of the product Hilbert space $\kil\times\hil$. Furthermore, $T$ is said to be closable if the closure $\overline{G(T)}$ is the graph of an operator. A densely defined linear operator $T$ is  closable if and only if its adjoint opreator  $T^*$ is densely defined, i.e.,  $\dom T^*$ is a dense linear subspace of $\hil$. In that case, the second adjoint operator $T^{**}$ of $T$ exists and its graph $G(T^{**})$ is just $\overline{G(T)}$. Densely defined closed operators are therefore   characterized  as being those closable operators $T$ for which the equality $T=T^{**}$ holds true. 

One of the most important results concerning densely defined closed operators is due to von Neumann: if $T$ is densely defined and closed  then both $T^*T$ and $TT^*$ are selfadjoint operators. Recently the authors proved the converse of this statement (\cite[Theorem 2.1]{SZ-TZS:reversed}): if $T^*T$ and $TT^*$ are both selfadjoint operators (provided that $T^*$ exists) then $T$ must be closed. Below we offer an improved version of this result:
\begin{theorem}\label{T:theorem61}
Let $T$ be a densely defined linear operator between the Hilbert spaces $\kil$ and $\hil$. Then the following assertions (i)-(vi) are equivalent:
\begin{enumerate}[\upshape (i)]
\item $T$ is closed.
\item $\ran M_{T^*,T}=\hil\times \kil$.
\item $\ran (I+T^*T)=\kil$ and $\ran (I+TT^*)=\hil$.
\item $T^*T$ and $TT^*$ are selfadjoint operators in the Hilbert spaces $\kil$ and $\hil$, respectively.
\item $G(T^*)^{\perp}\subseteq \ran M_{T^*,T}$.
\item $T^*$ is densely defined and $T^{**}=T$.
\end{enumerate}
\end{theorem}
\begin{proof}
The proof of implication (i)$\Rightarrow$(ii) is similar to that of implication (i)$\Rightarrow$(ii) of Theorem \ref{T:theorem31}, so it is left to the reader. Equivalence of (ii) and (iii) is obtained due formula 
\begin{equation*}
M_{T^*,T}\cdot M_{-T^*,-T}={\left(\begin{array}{cc}\!\! I+TT^* &  0\!\!\\ \!\! 0& I+T^*T\!\!\end{array}\right)}=M_{-T^*,-T}\cdot M_{T^*,T}.
\end{equation*}
Since $T^*T$ and $TT^*$ are positive symmetric operators, an immediate application of Theorem \ref{T:theorem55} shows that (iii) and (iv) are equivalent. Implications (ii)$\Rightarrow$(v),  (v)$\Rightarrow$(vi) (observe that $T^*\wedge T$ and use Theorem \ref{T:theorem22}) and (vi)$\Rightarrow$(i) are clear. The proof is therefore complete.
\end{proof}
We remark that the equivalence between (i) and (iv) has been established recently by A. Sandovici \cite{Sandovici} for linear relations. 
\section{Normal, unitary and orthogonal projection operators}
In this   section  we apply our foregoing results in order to gain some characterizations of normal and unitary operators as well as orthogonal projections.

To start with, we provide some formally weaker conditions implying normality of an operator $N$. Recall that a densely defined closed linear operator $N$ in a Hilbert space $\hil$ is called normal if the selfadjoint operators $N^*N$ and $NN^*$ are identical. In view of the characterization Theorem \ref{T:theorem61} of a closed densely defined operator, the definition of normality can be weakened by omitting the ``closedness'' assumption on $N$.  The ensuing theorem  says that it is moreover enough to assume $N^*N$ and $NN^*$ to be adjoint of each other:
\begin{theorem}
Let $N$ be a densely defined linear operator in a Hilbert space $\hil$. The following assertions (i)-(iv) are equivalent: 
\begin{enumerate}[\upshape (i)]
 \item $N$ is normal.
 \item $N^*N$ and $NN^*$ are selfadjoint operators such that $N^*N=NN^*$.
 \item $N^*N$ and $NN^*$ are adjoint of each other, i.e. $(N^*N)^*=NN^*$ and $(NN^*)^*=N^*N$.
 \item $N^*N$ and $NN^*$ are adjoint to each other such that 
 \begin{enumerate}[\upshape a)]
   \item $G(N^*N)^{\perp}\subseteq \ran (M_{N^*N,NN^*}),$
   \item $G(NN^*)^{\perp}\subseteq \ran (M_{NN^*,N^*N})$.
\end{enumerate}
\end{enumerate}
\end{theorem}
\begin{proof}
Implications (i)$\Leftrightarrow$(ii), (ii)$\Rightarrow$(iii) and (iii)$\Leftrightarrow$(iv) are obvious in view of Theorem \ref{T:theorem61} and Theorem \ref{T:theorem31}. The missing part (iii)$\Rightarrow$(ii) is obtained via the relations
\begin{equation*}
NN^*\subset (NN^*)^*=N^*N\subset (N^*N)^*=NN^*.
\end{equation*}
Here we used the fact that $N^*N$ and $NN^*$ are densely defined symmetric operators.
\end{proof}
We proceed to characterizations of unitary operators:
\begin{theorem}
Let $U$ be a linear operator in the Hilbert space $\hil$ such that $\ker U=\{0\}.$ Then the following assertions are equivalent: 
\begin{enumerate}[\upshape (i)]
\item $U$ is a unitary operator, i.e.,  $U$ is an everywhere defined bounded operator such that $U^*=U^{-1}$.
\item $U\wedge U^{-1}$ and $G(U)^{\perp}\subseteq \ran M_{U,U^{-1}}$.
\item $U$ is densely defined and $U^*=U^{-1}$.
\item $U$ is densely defined, closed with dense range such that $U^{-1}\subset U^*$.
\end{enumerate}
\end{theorem}
\begin{proof}
 From Theorem \ref{T:theorem22} we have (i)$\Rightarrow$(ii)$\Rightarrow$(iii)$\Rightarrow$(iv), so our only duty is to prove that (iv) implies (i). $U^{-1}$ exists as a densely defined closed operator, furthermore we have $\ran U\subseteq \dom U^*$, whence 
\begin{align}
\hil=\dom U^*+\ran U\subseteq \dom U^{*}.
\end{align}
This means that  $U^*$ is everywhere defined and continuous in account of the closed graph theorem. Consequently, $U$ is continuous too and $U^*=U^{-1}$. The proof is therefore complete.
\end{proof}
We close the paper by characterizing orthogonal projections (cf. also  \cite[Corollary 3.7]{Popovici}): 
\begin{theorem}
Let $P$ be a symmetric linear operator in a Hilbert space $\hil$. The following assertions (i)-(iii) are equivalent:
\begin{enumerate}[\upshape (i)]
 \item $P$ is an orthogonal projection, i.e.,  $P$  is an everywhere defined bounded operator such that $P=P^*=P^2$.
 \item $G(P)^{\perp}\subseteq \ran M_{P,P}\cap\ran M_{P,P^2}$.
 \item $P$ is selfadjoint and $P^2\subset P$.
\end{enumerate}
\end{theorem}
\begin{proof}
The equivalence between (ii) and (iii) is clear by Theorem \ref{T:theorem22}. Furthermore, (i) obviously implies (iii). Finally, if we assume (iii) then we infer that $P^2=PP^*=P$ because a selfadjoint operator has no proper selfadjoint extension. It remains to prove  that $P$ is continuous. Since have $\ran M_{P, P^*}=\hil\times\hil$ by selfadjointness, and also $\dom P^2=\dom P$, it follows that 
\begin{equation*}
\hil =\dom P+\ran P^*=\dom P+\ran P\subseteq \dom P,
\end{equation*}
 i.e.,  $\dom P=\hil$. By the closed graph we conclude that $P$ is bounded.
\end{proof}

\bibliographystyle{abbrv}

\end{document}